\documentclass[12pt,twoside]{amsart}
\usepackage{amsmath}
\usepackage{amsthm}
\usepackage{amsfonts}
\usepackage{amssymb}
\usepackage{latexsym}
\usepackage{mathrsfs}
\usepackage{amsmath}
\usepackage{amsthm}
\usepackage{amsfonts}
\usepackage{amssymb}
\usepackage{latexsym}
\usepackage{geometry}
\usepackage{dsfont}
\usepackage[dvips]{graphicx}
\usepackage{color}
\usepackage[all]{xy}

\date{}
\allowdisplaybreaks[4] \footskip=15pt
\renewcommand{\uppercasenonmath}[1]{}

\usepackage{graphicx,amssymb}
\usepackage[all]{xy}
\usepackage{amsmath}

\allowdisplaybreaks
\usepackage{amsthm}
\usepackage{color}

\theoremstyle{plain}
\newtheorem{theorem}{Theorem}[section]
\newtheorem{proposition}[theorem]{Proposition}

\newtheorem{corollary}[theorem]{Corollary}
\theoremstyle{definition}

\newtheorem{definition}[theorem]{Definition}

\theoremstyle{definition}

\theoremstyle{remark}
\newtheorem{remark}[theorem]{Remark}

\newcommand{\pf}{\noindent\begin {proof}}
\newcommand{\epf}{\end{proof}}

\newcommand{\Ker}{\mbox{\rm Ker}}
\newcommand{\Ext}{\mbox{\rm Ext}}
\newcommand{\Hom}{\mbox{\rm Hom}}
\newcommand{\Tor}{\mbox{\rm Tor}}

\newcommand{\E}{{\rm E}}

\def\Hom{{\rm Hom}}
\def\Ext{{\rm Ext}}
\def\Tor{{\rm Tor}}

\def\pd{{\rm pd}}
\def\Ker{{\rm Ker}}

\def\Im{{\rm Im}}

\def\Mod{{\rm Mod}}

\begin{document}
\begin{center}
{\large  \bf A note on Matlis localizations}

\vspace{0.5cm}  Xiaolei Zhang$^{a}$

{\footnotesize a. School of Mathematics and Statistics, Shandong University of Technology, Zibo 255000, China

E-mail: zxlrghj@163.com\\}
\end{center}

\bigskip
\centerline { \bf  Abstract}
\bigskip
\leftskip10truemm \rightskip10truemm \noindent

Let $R$ be a commutative ring and $S$ a  multiplicative subset of $R$. A ring $R$ is called  an $S$-Matlis ring if $\pd_RR_S\leq 1$.
In this note, we give some new  characterizations of $S$-Matlis rings in terms of $S$-strongly flat modules, $S$-weakly cotorsion modules and $S$-$h$-divisible modules. 
\vbox to 0.3cm{}

\leftskip0truemm \rightskip0truemm
\bigskip

\section{Introduction}

Throughout this paper, all rings are commutative with $1$. Let $R$ be a ring and $S$ be a multiplicative subset of $R,$ that is, $1\in S$ and $s_1s_2\in S$ for any $s_1\in S, s_2\in S.$ We always denote by $R_S$ the localization of $R$ at $S$.  A multiplicative set is called regular if it consists of non-zero-divisors. Note that if a multiplicative set $S$ is regular, then $R$ can be viewed as a subring of $R_S$ naturally.

In 2001, Fuchs and Salce in their famous monograph \cite{FS01} pointed out that, the projective dimension of the quotient field  of a domain $\leq 1$ has a tremendous influence on the entire module category.  A domain whose quotient field has projective dimension $\leq 1$ is said to be a Matlis domain, in honor of Eben Matlis, who first recognize the powerful consequences of this assumption (see \cite{M64}). Several characterizations of Matlis domains are given in \cite{BS02,FS01,M64} in terms of divisible modules, $h$-divisible modules, $h$-reduced modules, strongly flat modules and weakly cotorsion modules etc.

In 2017, Fuchs and Lee \cite{FS17} extended the notions of Matlis domains and some related modules to commutative rings with zero-divisors.  Let $R$ be a ring,  $R^{\times}$ the set of all non-zero-divisors in $R$ and $Q=R_{R^{\times}}$ its total ring of quotients.  A ring $R$ is called a Matlis ring if  $\pd_RQ\leq 1.$ An $R$-module $M$ is called  divisible if $sM=M$ for any $s\in R^{\times}$; and is called  $h$-divisible (resp., $h$-reduced) if the natural homomorphism $\Hom_R(Q,M)\rightarrow M$ is an epimorphism (resp., a zero morphism). Obviously, $h$-divisible modules are divisible. It was proved in \cite[Theorem 6.4]{FS17} that a ring $R$ is a Matlis ring if and only if $Q/R$ is a direct sum of countably presented submodules, if and only if every  divisible $R$-module is $h$-divisible. 

Recall that an $R$-module $M$ is called weakly cotorsion, denoted by $M\in \mathcal{WC}$, if $\Ext_R^1(Q,M)=0$; and is called strongly flat, denoted by $M\in \mathcal{SF}$, if $\Ext_R^1(M,N)=0$ for every $N\in \mathcal{WC}$. For the domain case, the strongly flat modules behave over Matlis domains nicely. One reason is that every strongly flat module over Matlis domains has projective dimension $\leq 1$; another reason is that  strongly flat modules over Matlis domains form a resolving class. Furthermore, these properties actually characterize Matlis domains (see \cite{M64,FST04}). The main motivation is to extend these two characterizations to commutative rings, or more generally commutative rings with regular multiplicative subsets (see \cite{aht05,BP19,WL11}). Actually, we obtain the following main result: 
\begin{theorem}$(=$ Theorem \ref{main}$)$ Let $R$ be a ring and $S$ a regular multiplicative subset of $R$. Then the following statements are equivalent:
	\begin{enumerate}
		\item $R$ is an $S$-Matlis ring;
		\item the class of $S$-strongly flat $R$-modules is resolving;
		\item the class of $S$-weakly cotorsion $R$-module is coresolving;
		\item the  cotorsion pair $(S\mbox{-}\mathcal{SF},S\mbox{-}\mathcal{WC})$ is hereditary;
		\item $\Ext_R^1(R_S/R,M)$ is $S$-$h$-reduced for any $($$S$-divisible$)$  $R$-module $M$;
		\item every $S$-strongly flat $R$-module has projective dimension at most $1$;
		\item every $S$-$h$-divisible $R$-module is $S$-weakly cotorsion.
	\end{enumerate}
\end{theorem}

\section{Preliminaries}

In this section we recall some notions and investigate some new properties to prove the main result of this note.
\begin{definition}\cite[Definition 1.6.10]{fk16}
Let $R$ be a ring and $S$ a  multiplicative subset of $R$. 
An $R$-module $M$ is said to be
\begin{enumerate}
	\item $S$-divisible if $sM=M$ for any $s\in S$;
	\item $S$-torsion-free if $sm=0$ with $s\in S$ and $m\in M$ whence $m=0$;
	\item $S$-torsion if for any $m\in M$, there is $s\in S$ such that $sm=0$. 
\end{enumerate}
\end{definition}
It is well-known that if a multiplicative subset $S$ of $R$  is  regular, then an $R$-module $M$ is $S$-divisible (resp., $S$-torsion-free) if and only if $\Ext_R^1(R/Rs,M)=0$ (resp., $\Tor^R_1(R/Rs,M)=0$).

 \begin{proposition}\label{w-gv-inj}
	Let $R$ be a ring and $S$ a regular multiplicative subset of $R$. Then an $R$-module $E$ is an $S$-divisible $S$-torsion-free $R$-module if and only if $E$ is an $R_S$-module.
\end{proposition}
\begin{proof} Suppose $E$ is an $R_S$-module. Assume that $sm=0$ with $s\in S$ and $m\in M$. Then $\frac{s}{1}m=0$ and so $m=\frac{1}{s}(\frac{s}{1}m)=0$, and hence $M$ is $S$-torsion-free. Let $s\in S$  and $f:Rs\rightarrow E$ be an $R$-homomorphism. Write $g:R\rightarrow E$ to be an $R$-homomorphism satisfying $g(a)=\frac{a}{s}f(s)$ for any $a\in R$. Then for any $b\in I$, we have $f(b)=\frac{s}{s}f(b)=\frac{1}{s}f(sb)=\frac{b}{s}f(s)=g(b)$. So $g$ is a lift of $f$, and hence $\Ext_R^1(R/Rs,M)=0$, that is, $E$ is $S$-divisible.
	
On the other hand, we only need to show: for any $s\in R$, the multiplication $m_s:E\xrightarrow{\times s} E$ is an isomorphism. Let $s\in S$. Since $E$ is $S$-torsion-free, so $se=0\in E$ implies $e=0$. Hence $m_s$ is a monomorphism. Since $E$ is $S$-divisible, we have $E/sE\cong\Ext_R^1(R/Rs,E)=0$ for any $s\in S$. Hence, $m_s$ is an epimorphism. Consequently, $E$ is an $R_S$-module.
\end{proof}

Let $S$ be a  multiplicative subset of a ring $R$. We say an ideal $I$ of $R$ is an $S$-ideal if $I\cap S\not=\emptyset.$  
\begin{definition}
	Let $R$ be a ring and $S$ a  multiplicative subset of $R$. 
	An $R$-module $M$ is said to be  $S$-injective if $\Ext_R^1(R/I,M)=0$ for any $S$-ideal $I$ of $R$.
\end{definition}

It is well-known that the quotient field $Q$ of a domain $D$ is both a flat module and an injective module over $D$. However, the localization $R_S$ is only a flat module  over a ring $R$. It is  not injective in general, but it is easy to verify that  $R_S$ is an $S$-injective module over $R$. 

\begin{theorem}\label{gv-inj-c}
	Let $R$ be a ring, $S$ a regular multiplicative subset of $R$ and $E$ an $R$-module. Then the following statements are equivalent:
	\begin{enumerate}
		\item $E$ is  $S$-injective;
		\item $\Ext_R^1(T,E)=0$ for any $S$-torsion module $T$;
		\item $\Ext_R^n(R/I,M)=0$ for any $S$-ideal $I$ of $R$ and any $n\geq 1$;
		\item $\Ext_R^n(T,E)=0$ for any $S$-torsion module $T$ and any $n\geq 1$.
	\end{enumerate}
\end{theorem}
\begin{proof}
	$(4)\Rightarrow (2)\Rightarrow (1)$ and 	$(4)\Rightarrow (3)\Rightarrow (1)$: Trivial.
	
	$(1)\Rightarrow (2)$: Similar with the proof of Baer's criterion for injective modules ( also see \cite[Proposition 2.3]{WL11}), and so we omit it.
	
	$(1)\Rightarrow (3):$ Let $E$ be an $S$-injective $R$-module and $I$ an $S$-ideal of $R$. Let $a\in I\cap S$. Then $a$ is a non-zero-divisor, and so $\pd_RR/Ra\leq 1$. Hence, $\Ext^n_R(R/Ra,E)=0$ for any $n\geq 1$. It follows by \cite[Proposition 4.1.4]{CE56} that $$ \Ext^1_{R/Ra}(R/I,\Hom_R(R/Ra,E))\cong \Ext^1_{R}(R/I,E)=0.$$ So $\Hom_R(R/Ra,E)$ is an injective
	$R/Ra$-module by Baer's criterion. It follows by \cite[Proposition 4.1.4]{CE56} again that, for any $n\geq 1$, we have $$\Ext^n_{R}(R/I,E)\cong \Ext^n_{R/Ra}(R/I,\Hom_R(R/Ra,E)) =0.$$
	
	$(3)\Rightarrow (4):$	It follows by $\Ext_R^n(R/I,E)\cong\Ext_R^{1}(R/I,\Omega_{n-1}(E))$ and 	$(1)\Rightarrow (2).$
\end{proof}
\begin{remark}
$(1)\Leftrightarrow (2)$ and  $(3)\Leftrightarrow (4)$ don't require $S$ to be regular. But $(2)\Rightarrow (3)$ needs  $S$ to be regular.

\end{remark}

Let $R$ be a ring, $S$ be a  multiplicative subset of $R$ and $M$ be an $R$-module. It follows by \cite[Theorem 3.4, Theorem 3.6]{WL11} that every $R$-module $M$ has an $S$-injective envelope:
$$\E_S(M):=\{x\in \E(M)\mid sx\in M\ \mbox{for some}\ s\in S\},$$ where $\E(M)$ is the injective envelope of $M$.  Moreover, $\E_S(M)/M$ is an $S$-torsion module.

Let $R$ be a ring and $S$ a  multiplicative subset of $R$. Let $M$ be an $R$-module. Denote by $$h_S(M)=\sum\limits_{f\in\Hom_R(R_S,M)}\Im(f).$$ 

\begin{definition}\cite{BP19}
	Let $R$ be a ring and $S$ a  multiplicative subset of $R$. 
	An $R$-module $M$ is said to be 
	\begin{enumerate}
		\item $S$-$h$-divisible if $h_S(M)=M$;
		\item $S$-$h$-reduced if $h_S(M)=0$. 
	\end{enumerate}
\end{definition}

It is easy to verify that  an $R$-module $M$  is $S$-$h$-divisible if and only if there is an epimorphism $R_S^{(\kappa)}\twoheadrightarrow M$; and $M$  is $S$-$h$-reduced if and only if $\Hom_R(R_S,M)=0.$ So $S$-$h$-divisible modules are closed under quotients, while $S$-$h$-reduced modules are closed under submodules. Trivially, if the multiplicative subset  $S$ is regular, then every $S$-injective module is $S$-divisible. Moreover, we have the following result.

\begin{proposition}\label{rinj-hd}
Let $R$ be a ring and $S$ a regular multiplicative subset of $R$. Then	every $S$-injective $R$-module is $S$-$h$-divisible.
\end{proposition}
\begin{proof} Let $E$ be an $S$-injective $R$-module. Then there is an epimorphism $\pi:F\twoheadrightarrow E$ with $F$ a free $R$-module. Considering the embedding map $i:F\hookrightarrow\E_S(F)$ where $\E_S(F)$ is the $S$-injective envelope of $F$, we have an epimorphic lift $\delta: \E_S(F)\twoheadrightarrow E$ of $\pi.$  We claim that $\E_S(F)$ is an $R_S$-module. We only need to show  $\E_S(F)$ is both $S$-divisible and $S$-torsion-free as an $R$-module by Proposition \ref{w-gv-inj}. Note that $\E_S(F)$ is $S$-injective, thus is $S$-divisible as $S$ is regular. Note that the $S$-injective envelope of $S$-torsion-free module is $S$-torsion-free. Since $S$ is regular, every free $R$-module $F$ is $S$-torsion-free, and hence $\E_S(F)$ is also $S$-torsion-free. It follows by  Proposition \ref{w-gv-inj} that  $\E_S(F)$ is an $R_S$-module. So there is an epimorphism $\sigma:R_S^{(\kappa)}\twoheadrightarrow \E_S(F)$. The composition $\delta\circ\sigma:R_S^{(\kappa)} \twoheadrightarrow E$ implies that $E$ is $S$-$h$-divisible.
\end{proof}

\begin{definition}\cite{LS19}
	Let $R$ be a ring and $S$ a  multiplicative subset of $R$. 
	An $R$-module $M$ is said to be 
	\begin{enumerate}
		\item $S$-weakly cotorsion, denoted by $S\mbox{-}\mathcal{WC}$, if $\Ext_R^1(R_S,M)=0$;
		\item $S$-strongly flat, denoted by $S\mbox{-}\mathcal{SF}$, if $\Ext_R^1(M,N)=0$ for any $S$-weakly cotorsion module $N$. 
	\end{enumerate}
\end{definition}

It follows by \cite[Lemma 1.2]{BP19} that an $R$-module $F$ is $S$-strongly flat if and only if it is a direct summand of an
$R$-module $G$ for which there exists an exact sequence of $R$-modules
$$0\rightarrow U\rightarrow G\rightarrow V \rightarrow 0$$
where $U$ is a free $R$-module and $V$ is a free $R_S$-module. Hence $$S\mbox{-}\mathcal{WC}=\Ker\Ext_R^1(R_S\mbox{-}\Mod,-).$$ It follows by \cite[Theorem 5.27,Theorem 6.11]{gt} that every $R$-module has an $S\mbox{-}\mathcal{WC}$-envelope and an special $S\mbox{-}\mathcal{SF}$-precover.

Recall that a pair $(\mathcal{A},\mathcal{B})$ of $R$-modules  is said to be a cotorsion pair if $\mathcal{A}=\Ker\Ext_R^1(-,\mathcal{B})$ and $\mathcal{B}=\Ker\Ext_R^1(\mathcal{A},-)$. Moreover, a cotorsion pair $(\mathcal{A},\mathcal{B})$ is called completed if for each $R$-module $M$ there is an exact sequence $0\rightarrow M\rightarrow B\rightarrow A\rightarrow0$ (or equivalently $0\rightarrow B\rightarrow A\rightarrow M\rightarrow0$) with $A\in \mathcal{A}$ and $B\in \mathcal{B}.$ A cotorsion pair $(\mathcal{A},\mathcal{B})$ is perfect if every $R$-module has an $\mathcal{A}$-cover and a $\mathcal{B}$-envelope. A cotorsion pair $(\mathcal{A},\mathcal{B})$ is hereditary if $\mathcal{A}$ is resolving (or equivalently $\mathcal{B}$ is coresolving).

It follows by \cite[Theorem 6.11]{gt} that  $(S\mbox{-}\mathcal{SF},S\mbox{-}\mathcal{WC})$ is a complete cotorsion pair. Recall from \cite[Definition 7.6]{BP19}  that a ring $R$ is called to be $S$-almost perfect if $R_S$ is a perfect ring and $R/sR$ is a perfect ring for every $s\in S.$ It was proved in  \cite[Theorem 7.9]{BP19} that a ring $R$ is  $S$-almost perfect if and only if $(S\mbox{-}\mathcal{SF},S\mbox{-}\mathcal{WC})$ is a perfect cotorsion pair, i.e., $S\mbox{-}\mathcal{SF}$ is covering and $S\mbox{-}\mathcal{WC}$ is enveloping. So it is natural to ask: 
\begin{center}
When the cotorsion pair $(S\mbox{-}\mathcal{SF},S\mbox{-}\mathcal{WC})$ is hereditary?
\end{center}
 In the domain case, Matlis \cite{M64} showed that, in a modern language, the classes of strongly flat modules and weak cotorsion modules consitute a hereditary cotorsion pair if and only if the basic domain $D$ is a Matlis domain, i.e., $\pd_DQ\leq 1$ where $Q$ is the quotient field of $D$.  One of the main motivation of this note is to extend this result to a more general situation. 
 
 The authors in \cite{aht05} gave several characterizations of rings satisfying $\pd_RR_S\leq 1$. For convenience, we call them $S$-Matlis rings:
 
\begin{definition} Let $R$ be a ring and $S$ a  multiplicative subset of $R$. Then $R$ is called an $S$-Matlis ring if $\pd_RR_S\leq 1$. 
\end{definition}

Several characterizations of $S$-Matlis rings are given in \cite{aht05}: 
\begin{theorem}\label{SM-ref}\cite[Theorem 1.1]{aht05} Let $R$ be a ring and $S$ a regular multiplicative subset of $R$. Then the following statements are equivalent:
	\begin{enumerate}
		\item $R$ is an $S$-Matlis ring;
		\item $R_S\oplus R_S/R$ is a tilting module;
		\item  the class of $S$-$h$-divisible modules is equal to $\Ker\Ext_R^1(R_S/R,-)$;
		\item the class of $S$-$h$-divisible modules is equal to that of $S$-divisible modules;
		\item $R_S/R$ is a direct sum of countably presented $R$-submodules;
		\item $R$ has an $S$-divisible envelope.
	\end{enumerate}
\end{theorem}

\section{Main result}

Now, we are ready to prove the main result of this note.

\begin{theorem}\label{main} Let $R$ be a ring and $S$ a regular multiplicative subset of $R$. Then the following statements are equivalent:
	\begin{enumerate}
		\item $R$ is an $S$-Matlis ring;
		\item the class of $S$-strongly flat $R$-modules is resolving;
		\item the class of $S$-weakly cotorsion $R$-module is coresolving;
		\item the  cotorsion pair $(S\mbox{-}\mathcal{SF},S\mbox{-}\mathcal{WC})$ is hereditary;
		\item $\Ext_R^1(R_S/R,M)$ is $S$-$h$-reduced for any $($$S$-divisible$)$  $R$-module $M$;
		\item every $S$-strongly flat $R$-module has projective dimension at most $1$;
		\item every $S$-$h$-divisible $R$-module is $S$-weakly cotorsion.
	\end{enumerate}
\end{theorem}

\begin{proof} 
	
	We always denote by $K:=R_S/R$. Then $K$ is $S$-torsion and $S$-$h$-divisible.
	
 $(1)\Rightarrow(3):$ Trivially, every injective module is $S$-weakly cotorsion and the class of  $S$-weakly cotorsion modules is closed under direct products. Now, let $0\rightarrow M\rightarrow N\rightarrow L\rightarrow0$ be a short exact sequence with $M$ and $N$   $S$-weakly cotorsion. Then $$0=\Ext_R^1(R_S,N)\rightarrow\Ext_R^1(R_S,L)\rightarrow\Ext_R^2(R_S,M)=0$$ is exact. So $\Ext_R^1(R_S,L)=0$, that is, $L$ is an $S$-weakly cotorsion module.
	
$(2)\Leftrightarrow(3)\Leftrightarrow(4):$	It follows by \cite[Lemma 5.24]{gt}.

$(3)\Rightarrow (5):$ Let $M$ be an $R$-module and $\E_S(M)$ its $S$-injective envelope. Then $\E_S(M)/M$ is $S$-torsion as an $R$-module. Moreover, we have a long exact sequence $0\rightarrow \Hom_R(K,M)\rightarrow\Hom_R(K,\E_S(M))\rightarrow\Hom_R(K,\E_S(M)/M)\rightarrow\Ext_R^1(K,M)\rightarrow\Ext_R^1(K,\E_S(M))=0.$ Setting
\begin{center}
$A=\Hom_R(K,\E_S(M))/\Hom_R(K,M)$ and $B=\Hom_R(K,\E_S(M)/M)$, 
\end{center}
 we have an exact sequence $$0\rightarrow A\rightarrow B\rightarrow \Ext_R^1(K,M)\rightarrow 0.$$ Since $$\Hom_R(R_S,B)=\Hom_R(R_S,\Hom_R(K,\E_S(M)/M))\cong \Hom_R(R_S\otimes_RK,\E_S(M)/M)=0,$$ $B$ is $S$-$h$-reduced.  It follows by \cite[Lemma 2.3]{FS09} that  $$\Ext_R^1(R_S,\Hom_R(K,\E_S(M)))\cong \Ext_R^1(R_S\otimes_RK,\E_S(M))$$
as $\Ext_R^1(K,\E_S(M))=0=\Tor_1^R(R_S,K)$(see Theorem \ref{gv-inj-c}).
Since $R_S\otimes_RK=0$, we have $\Ext_R^1(R_S,\Hom_R(K,\E_S(M)))=0,$ that is, $\Hom_R(K,\E_S(M))$ is $S$-weakly cotorsion. Since $B=\Hom_R(K,\E_S(M)/M)$ is $S$-$h$-reduced, so is its submodule $A$. Considering the exact sequence $$0\rightarrow\Hom_R(K,M)\rightarrow\Hom_R(K,\E_S(M))\rightarrow A\rightarrow0,$$  we have $\Hom_R(K,M)$ is  $S$-weakly cotorsion. It follows by $(3)$ that $A$ is also  $S$-weakly cotorsion, that is, $\Ext_R^1(R_S,A)=0$. Considering the exact sequence 
$$\Hom_R(R_S,B)\rightarrow\Hom_R(R_S,\Ext_R^1(K,M))\rightarrow\Ext_R^1(R_S,A),$$ we have $\Hom_R(R_S,\Ext_R^1(K,M))=0$, that is, $\Ext_R^1(K,M)$ is $S$-$h$-reduced.
	
$(5)\Rightarrow(1):$ Let $M$ be an $R$-module fitting into an exact sequence $$\Psi:\  0\rightarrow T\rightarrow M\xrightarrow{\pi} R_S\rightarrow0$$ where $T$ is $S$-torsion and $S$-$h$-divisible. Then $M$ is trivially $S$-divisible.
 
\textbf{Claim 1: $M$ is $S$-$h$-divisible.} Indeed, applying $\Hom_R(-,M)$ to the exact sequence $0\rightarrow R\rightarrow R_S\rightarrow K\rightarrow0$, we have an exact sequence $$0\rightarrow M/h(M)\rightarrow\Ext_R^1(K,M)\rightarrow\Ext_R^1(R_S,M)\rightarrow0.$$ By (5), $\Ext_R^1(K,M)$ is $S$-$h$-reduced, so is its submodule $M/h(M)$.  Since $T\subseteq h(M)$, we have an epimorphism $R_S\cong M/T\twoheadrightarrow M/h(M)$.
Since $M/h(M)$ is $S$-$h$-reduced and hence equal to $0$, and hence $M$ is $S$-$h$-divisible.

\textbf{Claim 2: $\Ext_R^1(R_S,T)=0$ for any  $S$-torsion and $S$-$h$-divisible $R$-module $T$.} Indeed, we just need to $\Psi$ splits. Since $M$ is $S$-$h$-divisible, there is an epimorphism $\rho: R_S^{(\kappa)}\twoheadrightarrow M$. Take $m_1\in M$ and $q_1\in R_S^{(\kappa)}$ such that $\pi(m_1)=1$ and $\rho(q_1)=m_1$. We claim that $T\oplus \rho(q_1R_S)=M.$ Let $m\in M-T$. Then $(m+T)s=(m_1+T)r$ for some  $s\in S.$ So there exists $t\in T$ such that $ms-m_1r=t$. Since $T$ is $S$-divisible, $t=t's$ for some $t'\in T$. Hence $(m-t')s=m_1r=\rho(q_1r)$ and $ms=\rho(q_1rs^{-1})s+t's$. Put $d=m-\rho(q_1rs^{-1})-t'$. Then $d\in M$ and $ds=0$. So $d\in T$ and $m\in T+\rho(q_1R_S)$. Next we will show $T\cap \rho(q_1R_S)=0.$ Indeed, let $t=\rho(q_1rs^{-1})\in T\cap \rho(q_1R_S).$ Then $ts=\rho(q_1r)=m_1r\in T$. So $r=\pi(ts)=0$. Hence the claim holds. It follows that $\Ext_R^1(R_S,T)=0$. 

Let $N$ be an arbitrary $R$-module and $\E_S(N)$ be its $S$-injective envelope. Then $\E_S(N)/N$ is an $S$-torsion $R$-module. It follows by Proposition \ref{rinj-hd} that $\E_S(N)/N$ is also an $S$-$h$-divisible $R$-module. By \textbf{Claim 2}, we have $\Ext_R^1(R_S,\E_S(N)/N)=0$. Since  $\E_S(N)$ is $S$-injective,  $\Ext_R^n(K,\E_S(N))=0$  for every $n\geq 1$ by Theorem \ref{gv-inj-c}. Considering the long exact sequence $$0=\Ext_R^{n}(R,\E_S(N))\rightarrow\Ext_R^{n+1}(K,\E_S(N))\rightarrow \Ext_R^{n+1}(R_S,\E_S(N))\rightarrow \Ext_R^{n+1}(R,\E_S(N))=0,$$ we have $\Ext_R^n(R_S,\E_S(N))=0$ for every $n\geq 2$.
Considering the long exact sequence $$0=\Ext_R^1(R_S,\E_S(N)/N)\rightarrow \Ext_R^2(R_S,N)\rightarrow\Ext_R^2(R_S,\E_S(N))=0,$$  we have $\Ext_R^2(R_S,N)=0.$ It follows that $\pd_RR_S\leq 1$, that is, $R$ is a Matlis ring.



 $(6)\Rightarrow(1):$ Trivial.
 
  $(1)\Rightarrow(6):$ It follows by \cite[Lemma 1.2]{BP19} that every $S$-strongly flat module is a direct summand of $M$ that fits into a short exact sequence $0\rightarrow F\rightarrow M\rightarrow G\rightarrow0$ where $F$ is $R$-free and $G$ is $\{R_S\}$-free. Since $\pd_RR_S\leq 1$, so $\pd_RM\leq 1$. Hence every $S$-strongly flat $R$-module has projective dimension at most $1$.
  


$(1)\Leftrightarrow (7)$ Let $M$ be an $S$-$h$-divisible module. Consider the exact sequence $\Hom_R(R_S,M)\twoheadrightarrow \Hom_R(R,M)\rightarrow \Ext_R^1(R_S/R,M)\rightarrow\Ext_R^1(R_S,M)\rightarrow \Ext_R^1(R,M)=0$. Then  $\Ext_R^1(R_S/R,M)=0$ is equivalent to $\Ext_R^1(R_S,M)=0$. Hence the equivalence follows by Theorem \ref{SM-ref}.
\end{proof}

If $S$ consists of all non-zero-divisors in Theorem \ref{main}, then we can give the following characterizations of Matlis rings:
\begin{corollary} Let $R$ be a ring with $Q$ its total ring of quotients. Then the following statements are equivalent:
\begin{enumerate}
\item $R$ is a Matlis ring;
\item the class of strongly flat $R$-modules is resolving;
\item the class of weakly cotorsion $R$-module is coresolving;
\item the cotorsion pair $(\mathcal{SF},\mathcal{WC})$ is hereditary;
\item $\Ext_R^1(Q/R,M)$ is $h$-reduced for any $($divisible$)$  $R$-module $M$;
\item every strongly flat $R$-module has projective dimension at most $1$;
\item every $h$-divisible $R$-module is weakly cotorsion.
	\end{enumerate}
\end{corollary}

\end{document}